\documentclass{amsart}
\usepackage{amsfonts,amssymb,mathrsfs,amsthm}

\usepackage{graphicx}
\usepackage[all]{xy}

\newtheorem{thm}{Theorem}[section]

\newtheorem{lem}[thm]{Lemma}

\newtheorem{case}{Case}

\theoremstyle{definition}

\theoremstyle{remark}
\newtheorem*{rmk}{Remark}

\begin{document}

\title[A generalization of Menon's identity with Dirichlet characters]{A generalization of Menon's identity with Dirichlet characters}

\author{Yan Li}

\address{Department of Applied Mathematics, China Agricultural
university, Beijing 100083, China} \email{liyan\_00@cau.edu.cn,\ liyan\_00@mails.tsinghua.edu.cn}

\author{Xiaoyu Hu}
\address{Department of Applied Mathematics, China Agricultural
university, Beijing 100083, China}\email{hxyyzptx@126.com}

\author{Daeyeoul Kim*}

\address{Department of Mathematics and Institute of Pure and Applied Mathematics \\
 Chonbuk National University  \\
  567 Baekje-daero, Deokjin-gu, Jeonju-si, Jeollabuk-do 54896\\
South Korea} \email{kdaeyeoul@jbnu.ac.kr}

 \thanks{*Corresponding author}

\begin{abstract}
The classical Menon's identity \cite{Menon} states that
 \begin{equation*}\label{oldbegin1}
\sum_{\substack{a\in\Bbb Z_n^\ast
}}\gcd(a -1,n)=\varphi(n) \sigma_{0} (n),
\end{equation*}
where for a positive integer $n$, $\Bbb Z_n^\ast$ is the group of units of the ring $\Bbb Z_n=\Bbb Z/n\Bbb Z$, $\gcd(\ ,\ )$ represents the greatest common divisor, $\varphi(n)$
is the Euler's totient function and $\sigma_{k} (n) =\sum_{d|n } d^{k}$ is the divisor function.
 In this paper, we generalize Menon's identity with Dirichlet characters in the following way:
\begin{equation*}
 \sum_{\substack{a\in\Bbb Z_n^\ast \\ b_1, ..., b_k\in\Bbb Z_n}}
 \gcd(a-1,b_1, ..., b_k, n)\chi(a)=\varphi(n)\sigma_k\left(\frac{n}{d}\right),
\end{equation*}
where $k$ is a non-negative integer
and $\chi$ is a Dirichlet character modulo $n$ whose conductor is $d$.
 Our result can be viewed as an extension of Zhao and Cao's result \cite{Z-C} to $k>0$. 
 It can also be viewed as an extension of Sury's result \cite{Sury} to Dirichlet characters.

\end{abstract}

\maketitle

\section{Introduction}
There is a beautiful identity due to P. K. Menon \cite{Menon}, which states that, for
any positive integer $n$, we have
 \begin{equation}\label{oldbegin1}
\sum_{\substack{a\in\Bbb Z_n^\ast
}}\gcd(a -1,n)=\varphi(n) \sigma_{0} (n),
\end{equation}
where $\Bbb Z_n^\ast$ is the group of units of the ring $\Bbb Z_n=\Bbb Z/n\Bbb Z$, $\gcd(\ ,\ )$ represents the greatest common divisor,  $\varphi$
is the Euler's totient function and $\sigma_{k} (n) =\sum_{d|n } d^{k}$ is the divisor function.

In 2009, B. Sury \cite{Sury} obtained the following Menon-type identity
 \begin{equation}\label{oldbegin2}
  \sum_{\substack{a\in\Bbb Z_n^\ast \\ b_1, ..., b_k\in\Bbb Z_n}}
 \gcd(a-1,b_1, ..., b_k, n)=\varphi(n)\sigma_k(n)
\end{equation}
 by using Cauchy-Frobenius-Burnside lemma.
 Miguel \cite{Mig1}, \cite{Mig2}
 extended identities \eqref{oldbegin1} and \eqref{oldbegin2} from
 $\Bbb Z$ to any residually finite Dedekind domain.

 Recently,  Zhao and Cao \cite{Z-C} derived the following elegant Menon-type identity with Dirichlet characters
 \begin{equation}\label{char1}
\sum_{\substack{a\in\Bbb Z_n^\ast
}}
 \gcd(a-1,  n)\chi(a)=\varphi(n) \sigma_{0} \left(\frac{n}{d}\right),
\end{equation}
where $\chi$ is a Dirichlet character modulo $n$ and
$d$ is the conductor of $\chi$.

T\'{o}th \cite{To3} generalized the identity \eqref{char1} from gcd functions to even functions $({\rm mod}\ n)$. 
In \cite{LK2}, Li and Kim obtained another Menon-type identity by replacing Dirichlet characters of $\Bbb Z_n^\ast$  in \eqref{char1} with additive characters of $\Bbb Z_n$.
For other generalizations of Menon's identity, see \cite{H1}, \cite{HW2},
\cite{HW3}, \cite{LK}, 
\cite{LK1}, 
\cite{Tar} and \cite{To1}.

Denote
\begin{equation}\label{begin1}
 S_\chi(n,k)=\sum_{\substack{a\in\Bbb Z_n^\ast \\ b_1, ..., b_k\in\Bbb Z_n}}
 \gcd(a-1,b_1, ..., b_k, n)\chi(a).
\end{equation}
In this article, we will explicitly compute $S_\chi(n,k)$.

Our main result is the following theorem.
\begin{thm}\label{thm1}
Let $n$ be a positive integer and $\chi$ be a Dirichlet character modulo $n$ whose conductor is $d$. 
Assume $k$ is a non-negative integer. 
Then, we have the following identity:
\begin{equation}\label{begin2}
 \sum_{\substack{a\in\Bbb Z_n^\ast \\ b_1, ..., b_k\in\Bbb Z_n}}
 \gcd(a-1,b_1, ..., b_k, n)\chi(a)=\varphi(n)\sigma_k\left(\frac{n}{d}\right).
\end{equation}
\end{thm}
\begin{rmk}
 If $\chi$ is the trivial character, then 
 \eqref{begin2} reduces to Sury's identity \eqref{oldbegin2}.
  Further, if $k=0$, then our identity 
  \eqref{begin2} reduces to Zhao and Cao's identity \eqref{char1}.
\end{rmk}
The rest of paper is organized as follows. In section 2, we prove Theorem \ref{thm1} in the special  case of $n$ being a prime power. 
The general case is treated in section 3 by combining the prime power cases with the Chinese remainder theorem.

\section{Prime power case}
In this section, we  assume $n=p^m$, where $p$ is a prime number and $m$ is a positive integer.
 Let $\chi$ be a Dirichlet character modulo $n$ with conductor $d$. 
 Since $d\mid n$, we denote $d=p^t$, where $0\leq t\leq m$.\

Our proof of Theorem \ref{thm1} in this case can be viewed as a combination of techniques of \cite{LK}  and \cite{Z-C}. \

Similarly as \cite{LK} (see p.46), 
we shall introduce filtrations for the additive group $\Bbb Z_n$ and 
the multiplicative group $\Bbb Z_n^\ast$, respectively.
 Since $n=p^m$ is a prime power,
 the whole subgroups of $\Bbb Z_n$ form a chain:
\begin{equation}\nonumber
0=p^m\Bbb Z_n\subset p^{m-1}\Bbb Z_n\subset ...\subset p\Bbb Z_n\subset \Bbb Z_n .
\end{equation}
The multiplicative group $\Bbb Z_n^\ast$ also has a filtration consisting of multiplicative subgroups:
\begin{equation}\nonumber
1=1+p^m\Bbb Z_n\subset 1+p^{m-1}\Bbb Z_n\subset ...\subset 1+p\Bbb Z_n\subset \Bbb Z_n^\ast.
\end{equation}
For simplicity of the proof, we introduce the following notations.
\begin{align*}
R_j&=p^j\Bbb Z_n \ \ {\rm  with} \ \ 0\leq j\leq m\ \  {\rm  and}\ \  R_{m+1}=\varnothing,\\
U_0&=\Bbb Z_n^\ast, U_i=1+p^i\Bbb Z_n\ \ {\rm  with} \ \ 1\leq i\leq m\ \  {\rm  and}\ \  U_{m+1}=\varnothing,\\
S_j&=R_j-R_{j+1} \ \ {\rm  with} \ \  0\leq j\leq m,\\
V_i&=U_i-U_{i+1} \ \ {\rm  with} \ \  0\leq i\leq m.
\end{align*}
Clearly, $\Bbb Z_n=\bigcup \limits_{j=0}^m S_j$ and $\Bbb Z_n^\ast=\bigcup \limits_{i=0}^m V_i$ with disjoint union.
 Also, we have
\begin{align}\nonumber
\# U_0&=p^m-p^{m-1}, \   \# U_{m+1}=0 \   {\rm  and}\\
\# U_i&=p^{m-i} \ \ {\rm  with}\ \  1\leq i\leq m,\nonumber
\end{align}
where \# denotes the cardinality of sets.\

Consider
\begin{align}\label{LiAdd3}
S_{\chi}(p^m,k)&=\sum_{\substack{a\in\Bbb Z_n^\ast \\ b_1, ..., b_k\in\Bbb Z_n}}
 \gcd(a-1,b_1, ..., b_k, n)\chi (a)\\
&=\sum_{s=0}^m \sum_{\substack{\gcd (b_1, ..., b_k, n)=p^s \\ b_1, ..., b_k\in\Bbb Z_n}} \sum_{a\in \Bbb Z_n^\ast} \gcd (a-1, p^s) \chi (a)\nonumber\\
&=\sum_{s=0}^m\left(\sum_{a\in \Bbb Z_n^\ast} \gcd (a-1, p^s) \chi (a)\right) \left(\sum_{\substack{\gcd (b_1, ..., b_k, n)=p^s \\ b_1, ..., b_k\in\Bbb Z_n}}1\right).\nonumber
\end{align}
Therefore, we need to compute
$$\sum_{a\in \Bbb Z_n^\ast} \gcd (a-1, p^s) \chi (a)\ \  {\rm  and}\  \sum_{\substack{\gcd (b_1, ..., b_k, n)=p^s \\ b_1, ..., b_k\in\Bbb Z_n}}1$$
explicitly. These will be done in Lemma \ref{2.2} and Lemma \ref{thm3}, respectively.\

Next, we will state a lemma, which is key to the proof of Lemma \ref{2.2}.
\begin{lem}\label{thm2}
Let $n=p^m$ and $\chi$ be a Dirichlet character modulo $n$ with conductor $p^t$, where $0\leq t\leq m$. Then, for $0\leq i\leq m$, we have
\begin{equation}\nonumber
\sum_{a\in U_i} \chi(a)=\left\{
						\begin{aligned}
						&\# U_i, \ \ if\ \  i=t, t+1,..., m.\\
						&0,\ \  otherwise.
						\end{aligned}
				    \right.
\end{equation}
\end{lem}
\begin{proof}
By the definition of conductor, $p^t$ is the smallest integer such that $\chi$ factors through $\Bbb Z_{p^t}^\ast$. 
Therefore, $\chi$ is trivial on $U_t$, 
but nontrivial on $U_{t-1}$ if $t\geq 1$. Since $U_m \subset ... \subset U_1 \subset U_0 $ forms a filtration of $\Bbb Z_n^\ast$, $\chi$ is trivial on $U_i$ for $t\leq i\leq m$ and nontrivial on other $U_i$. 
Hence,
$$\sum_{a\in U_i} \chi (a)=\# U_i, \ {\rm for}\ \  i=t,t+1,.., m.$$
For $i=0,..., t-1,$ the restriction of $\chi$ on $U_i$ is a nontrivial character for the multiplicative group $U_i$.
 By the orthogonality of characters, we have
$$\sum_{a\in U_i} \chi (a)=0, \ {\rm for}\ \  i=0,1,...,t-1.$$
\end{proof}
\begin{lem}\label{2.2}
Let $n=p^m$ and $\chi$ be a Dirichlet character modulo $n$ with conductor $p^t$, where $0\leq t\leq m$. 
Let $s$ be an integer such that $0\leq s\leq m$. Then,
\begin{equation}\label{LiAdd2}
\sum_{a\in \Bbb Z_n^\ast} \gcd(a-1,p^s) \chi(a)=\left\{
						\begin{aligned}
						(s-t+1)&(p^m-p^{m-1}), \  &if \ s\geq t,\\
						&0,\ \ \ \ \  &otherwise.
						\end{aligned}
				    \right.
\end{equation}
\end{lem}
\begin{proof}
By definitions of $U_i$ and $V_i$, we know that $\Bbb Z_n^\ast =\bigcup_{i=0}^m V_i$ with disjoint union
 and $\gcd (a-1, p^m)=p^i$ if $a\in V_i$, where $0\leq i\leq m$.\
Therefore, we have
\begin{align*}
&\sum_{a\in \Bbb Z_n^\ast} \gcd(a-1, p^s)\chi (a)\\
=&\sum_{i=0}^m \sum_{a\in V_i}\gcd(a-1, p^s)\chi (a)\\
=&\sum_{i=0}^{s-1} \sum_{a\in V_i} p^i \chi (a)+\sum_{i=s}^m \sum_{a\in V_i}p^s \chi (a)\\
=&\sum_{i=0}^{s-1}p^i\left(\ \sum_{a\in U_i}\chi (a)- \sum_{a\in U_{i+1}}\chi (a)\right)+p^s\sum_{a\in U_s}\chi(a).
\end{align*}
The last equality is due to that
$$V_i=U_i-U_{i+1}\  {\rm and}\  U_s=\bigcup^m_{i=s}V_i \  {\rm  with\ disjoint \ union.}$$
Changing the summation index, we get
\begin{align}\label{LiAdd1}
&\sum_{i=0}^{s}p^i\sum_{a\in U_i}\chi (a)-\sum_{i=1}^{s}p^{i-1}\sum_{a\in U_i}\chi (a)\nonumber\\
=&\sum_{a\in U_0}\chi(a)+\sum_{i=1}^s(p^i-p^{i-1})\sum_{a\in U_i}\chi(a).
\end{align}
In the following, we calculate the left hand side of \eqref{LiAdd2} case by case.
\begin{case}$t=0$\

In this case, $\chi$ is a trivial character of $\Bbb Z_n^\ast$.\

According to \eqref{LiAdd1},
\begin{align*}
\sum_{a\in \Bbb Z_n^\ast} \gcd(a-1,p^s) \chi(a)&
=\# U_0+\sum_{i=1}^s(p^i-p^{i-1})\# U_i\\
&=(p^m-p^{m-1})+\sum_{i=1}^s(p^i-p^{i-1})p^{m-i}\\
&=(s+1)(p^m-p^{m-1}).
\end{align*}
\end{case}

\begin{case}$s\geq t$ and $t\geq 1$.\\
By equation \eqref{LiAdd1} and Lemma \ref{thm2}, we have
\begin{align*}
\sum_{a\in \Bbb Z_n^\ast} \gcd(a-1,p^s) \chi(a)&=\sum_{i=t}^s(p^i-p^{i-1})\# U_i\\
&=\sum_{i=t}^s(p^i-p^{i-1})p^{m-i}\\
&=(s-t+1)(p^m-p^{m-1}).
\end{align*}
\end{case}
\begin{case}$s<t$

Note in this case, $t\geq 1$. According to equation \eqref{LiAdd1} and Lemma \ref{thm2}
\begin{align*}
\sum_{a\in \Bbb Z_n^\ast} \gcd(a-1,p^s) \chi(a)&=0+\sum_{i=1}^s(p^i-p^{i-1})*0=0.
\end{align*}
\end{case}
Case 1 and Case 2 can be unified, which leads to the final result.
\end{proof}
\begin{rmk}
In Lemma \ref{2.2}, if $s=m$, this is just Lemma 3.1 of \cite{Z-C}.
 In fact, for $s\geq t$, $\chi$ can also be viewed as a Dirichlet character modulo $p^s$ with conductor $p^t$.
  In this case,
$$\sum_{a\in \Bbb Z_{p^m}^\ast} \gcd (a-1, p^s) \chi (a)=p^{m-s}\sum_{a\in \Bbb Z_{p^s}^\ast} \gcd (a-1, p^s) \chi (a).$$
Therefore, Lemma \ref{2.2} can be deduced from Lemma 3.1 of \cite{Z-C} in case of $s\geq t$. 
However, here, we give a unified proof including all cases.
\end{rmk}
\begin{lem}\label{thm3}
Let $n=p^m$ be a prime power and $s\geq 0$ be an integer. 
Assume $k\geq 0$ is an integer.
 Then
\begin{equation}\nonumber
\sum_{\substack{b_1, ..., b_k\in \Bbb Z_n \\ \gcd(b_1, ..., b_k,p^m)=p^s}} 1=\left\{
						\begin{aligned}
						p^{(m-s)k}-&p^{(m-s-1)k}, \ &if\   s< m,\\
						&1,\  &if\  s=m.
						\end{aligned}
				    \right.
\end{equation}
\end{lem}

\begin{proof}
The case $k=0$ is obvious. Thus, we assume $k\geq 1.$\

Note that $p^s\mid \gcd(b_1, ..., b_k, p^m)$ if and only if $b_1, ..., b_k\in p^s \Bbb Z_n$ holds. 
Therefore, for $0\leq s\leq m-1$,
$$\gcd(b_1, ..., b_k, p^m)=p^s \text{ if and only if } (b_1,..., b_k)\in (p^s \Bbb Z_n)^k-(p^{s+1}\Bbb Z_n)^k.$$
Clearly, for $s=m$,
$$\gcd (b_1,..., b_k, p^m)=p^s \text{ if and only if } (b_1,..., b_k)\in (p^s\Bbb Z_n)^k.$$
Since $\#(p^s \Bbb Z_n)=p^{m-s}$ for $0\leq s\leq m$, we get the desired result.\
  \end{proof}

 Finally, we prove the following result, which is a special case of Theorem \ref{thm1}.
 \begin{thm}\label{thhh3}
 Let $n=p^m$ be a prime power and $\chi$ be a Dirichlet character whose conductor is $d=p^t$.
  Assume $k$ is a non-negative integer. Then, the following identity holds
 $$ \sum_{\substack{a\in\Bbb Z_n^\ast \\ b_1, ..., b_k\in\Bbb Z_n}}
 \gcd(a-1,b_1, ..., b_k, n)=\varphi(n)\sigma_k\left(\frac{n}{d}\right).
$$
 \end{thm}

\begin{proof}
 By equation \eqref{LiAdd3}, we have
 \begin{equation}\label{LiAdd4}
S_\chi (p^m, k)=\sum_{s=0}^m\left(\sum_{a\in \Bbb Z_n^\ast} \gcd (a-1, p^s) 
\chi (a)\right)\left(\sum_{\substack{b_1, ..., b_k\in \Bbb Z_n \\ \gcd(b_1, ..., b_k,p^m)=p^s}}1\right).
\end{equation}
Substituting Lemma \ref{2.2} into \eqref{LiAdd4}, we get
\begin{equation}\label{LiAdd5}
S_\chi (p^m, k)=\sum_{s=t}^m(s-t+1)(p^m-p^{m-1})
\left(\sum_{\substack{b_1, ..., b_k\in \Bbb Z_n \\ \gcd(b_1, ..., b_k,p^m)=p^s}}1\right) .
\end{equation}
It follows from Lemma \ref{thm3} and \eqref{LiAdd5} that
\begin{align}
S_\chi (p^m, k)=&\varphi(p^m)\left(\sum_{s=t}^{m-1}(s-t+1)(p^{(m-s)k}-p^{(m-s-1)k})+(m-t+1)\right)\\
=&\varphi(p^m)\left(\sum_{s=t}^{m}(s-t+1)p^{(m-s)k}-\sum_{s=t}^{m-1}(s-t+1)p^{(m-s-1)k}\right)\nonumber\\
=&\varphi(p^m)\left(\sum_{s=t}^{m}(s-t+1)p^{(m-s)k}-\sum_{s=t+1}^{m}(s-t)p^{(m-s)k}\right) .\nonumber
\end{align}
The last equality is by substituting $s+1$ with $s$ in the posterior summation.
 Hence
\begin{align}
S_\chi (p^m, k)=&\varphi(p^m)\left(p^{(m-t)k}+\sum_{s=t+1}^{m}p^{(m-s)k}\right)\nonumber\\
=&\varphi(p^m)\left(p^{(m-t)k}+\sum_{s=0}^{m-t-1}p^{sk}\right) .\nonumber
\end{align}
The last equality can be derived from substituting $m-s$ with $s$.
 Therefore,
\begin{equation}
S_\chi (p^m, k)=\varphi(p^m)\sum_{s=0}^{m-t}p^{sk}=\varphi(p^m)\sigma_k\left(\frac{p^m}{p^t}\right)  \nonumber
\end{equation}
which concludes the proof.
\end{proof}

\section{The general case}
In this section, we will prove the main  theorem.
 First, we show $S_\chi(n,k)$ is multiplicative with respect to $n$,
  by the Chinese remainder theorem. Then, using multiplicative property, 
  we prove Theorem \ref{thm1} by combining prime power cases, which are treated in section 2.

Let $n=n_1 n_2$ be the product of positive integers $n_1$ and $n_2$ such that 
$\gcd(n_1,n_2)=1$. By the Chinese remainder theorem, we have the ring isomorphism: 
$\Bbb Z_n \simeq \Bbb Z_{n_1}\oplus \Bbb Z_{n_2}$, which induces the multiplicative group isomorphism: 
$\Bbb Z_n^\ast \simeq \Bbb Z_{n_1}^\ast\times \Bbb Z_{n_2}$. 
Therefore, each Dirichlet character modulo $n$ can be uniquely written as $\chi=\chi_1\cdot\chi_2$,
 where $\chi$, $\chi_1$ and $\chi_2$ are Dirichlet characters modulo
  $n$, $n_1$ and $n_2$, respectively. 
  Explicitly,
\begin{equation}\nonumber
\chi(c\  {\rm  mod}\  n)=\chi_1(c\ {\rm  mod}\  n_1)\cdot \chi_2(c\ {\rm  mod}\  n_2)
\end{equation}
for any integer $c$ such that $\gcd (c,n)=1.$\

To simplify notations, for $a\in \Bbb Z_n$, we let $a'\in \Bbb Z_{n_1}$ and $a''\in \Bbb Z_{n_2}$ denote the image of
 $a$ in $\Bbb Z_{n_1}$ and $\Bbb Z_{n_2}$, respectively, i.e. $a'\equiv a\mod n_1$ and $a''\equiv a\mod n_2$.
  Let $d, d_1$ and $d_2$ be the conductors of $\chi, \chi_1$ and $\chi_2$, respectively. 
  It is well known that $d=d_1 d_2$.

The following lemma shows that $S_\chi(n,k)$ is multiplicative.
\begin{lem}\label{lem3}
Notations as above, we have
$$S_\chi(n,k)=S_{\chi_1}(n_1,k)\cdot S_{\chi_2}(n_2,k).$$
\end{lem}
\begin{proof}  First, we check that
\begin{align*}
&\sum_{\substack{a\in\Bbb Z_n^\ast \\ b_1, ..., b_k\in\Bbb Z_n}} \gcd(a-1,b_1, ..., b_k, n)\chi (a)\\
=&\sum_{\substack{a\in\Bbb Z_n^\ast \\ b_1, ..., b_k\in\Bbb Z_n}} \gcd(a-1,b_1, ..., b_k, n_1) \gcd(a-1,b_1, ..., b_k, n_2)\chi_1(a)\chi_2(a) \\
=&\sum_{\substack{a'\in\Bbb Z_{n_1}^\ast \\ b'_1, ..., b'_k\in\Bbb Z_{n_1}}}\gcd(a'-1,b'_1, ..., b'_k, n_1)\chi_1(a')\\\times &\sum_{\substack{a''\in\Bbb Z_{n_2}^\ast \\ b''_1, ..., b''_k\in\Bbb Z_{n_2}}}\gcd(a''-1,b''_1, ..., b''_k, n_2)\chi_2(a'').
\end{align*}
The last equality is by Chinese remainder theorem. 
Indeed, as $(a, b_1,...,b_k)$ runs over $\Bbb Z_n^\ast\times(\Bbb Z_{n})^k, (a', b'_1,..., b'_k, a'',b''_1,...,b''_k)$ runs over $\Bbb Z_{n_1}^\ast\times(\Bbb Z_{n_1})^k\times\Bbb Z_{n_2}^\ast\times(\Bbb Z_{n_2})^k$, too. 
Therefore, we have
$$S_\chi(n,k)=S_{\chi_1}(n_1,k)\cdot S_{\chi_2}(n_2,k).$$
\end{proof}
\begin{rmk}
The proof of Lemma \ref{lem3} is similar to that of Lemma 2.1 in \cite{LK}. 
Also see the proof of Theorem \ref{thm1} and Theorem 1.2 in \cite{Z-C}.
\end{rmk}
{\bf{Proof of Theorem \ref{thm1}} :}
Let $n=p_1^{m_1}p_2^{m_2}\cdots p_u^{m_u}$ be the prime factorization of $n$. 
Then, $\chi$ can be uniquely written as $\chi= \chi_1\chi_2\cdots\chi_u$, where $\chi_i$ is a Dirichlet character modulo $p_i^{m_i}$ with conductor $d_i$, for $1\leq i\leq u$.
 It is well known that $d=d_1d_2\cdots d_u$. Denote $n_i=p_i^{m_i}.$

Finally, Theorem \ref{thhh3}  and  Lemma \ref{lem3} yield
$$S_\chi(n,k)=\prod_{i=1}^u S_{\chi_i}(n_i,k)=\prod_{i=1}^u\varphi(n_i)\sigma_k\left(\frac{n_i}{d_i}\right).$$
Since the arithmetic functions $\varphi$ and $\sigma_k$ are multiplicative, we have
$$S_\chi(n,k)=\varphi(n)\sigma_k\left(\frac{n}{d}\right).  \quad\quad\quad \qed$$


\begin{thebibliography}{[20]}

\bibitem{H1} P. Haukkanen,  \textit{Menon's identity with respect to a generalized divisibility relation}, Aequationes Math. 70 (3) (2005) 240--246.

\bibitem{HW2} P. Haukkanen, J. Wang, \textit{High degree analogs of Menon's identity}, Indian J. Math. 39 (1) (1997) 37--42.
\bibitem{HW3} P. Haukkanen, J. Wang, \textit{A generalization of Menon's identity with respect to a set of polynomials}, Portugal. Math., 53 (3) (1996), 331--337.


\bibitem{LK} Y. Li, D. Kim,  \textit{A Menon-type  identity with many tuples of  group of units in residually finite Dedekind domains}, J. Number Theory 175 (2017), 42-–50.


\bibitem{LK1} Y. Li, D. Kim,  \textit{Menon-type identities derived from actions of subgroups of general linear groups}, J.  Number Theory 179 (2017) 97--112.

\bibitem{LK2} Y. Li, D. Kim,  \textit{Menon-type identities with additive characters}, submitted.

\bibitem{Menon} P. K. Menon, \textit{On the sum $\sum (a-1,n)[(a,n)=1]$},
J. Indian Math. Soc. (N.S.)  29 (1965),  155--163.

\bibitem{Mig1}  C. Miguel, \textit{Menon's identity in residually finite Dedekind domains}, J. Number Theory  137 (2014), 179--185.



\bibitem{Mig2}  C. Miguel, \textit{A Menon-type identity in residually finite Dedekind domains}, J. Number Theory  164 (2016), 43--51.



\bibitem{Rama} V. S.  Ramaiah, \textit{Arithmetical sums in regular convolutions}, J. Reine Angew. Math. 303/304 (1978), 265--283.

\bibitem{Ram}
S. Ramanujan, \textit{On Certain Trigonometric Sums and their Applications in the Theory of Numbers}, Transactions of the Cambridge Philosophical Society, 22 (15), (1918) 259-–276
 (pp. 179-–199 of his Collected Papers).
















\bibitem{Sury} B. Sury, \textit{Some number-theoretic identities from group actions},
Rendiconti del Circolo Matematico di Palermo 58 (2009), 99--108.




\bibitem{Tar} M. T\u{a}rn\u{a}uceanu, \textit{A generalization of Menon’s identity}, J. Number Theory,  132 (2012), 2568--2573.


\bibitem{To1} L. T\'{o}th, \textit{Menon's identity and arithmetical sums representing functions of several variables}, Rend. Semin. Mat. Univ. Politec. Torino 69 (1) (2011) 97--110.


\bibitem{To3} L. T\'{o}th, \textit{Menon-type identities concerning Dirichlet
characters}, https://arxiv.org/pdf /1706.03478.pdf

\bibitem{Z-C} Zhao, X.-P., Z-F. Cao,   \textit{Another generalization of Menon's identity},
Int. J. Number Theory 13 (2017), no. 9, 2373–-2379.
\end{thebibliography}
\end{document}